\documentclass[ ]{elsarticle}

\journal{Journal of Computational and Applied Mathematics}

\usepackage[letterpaper,margin=1in]{geometry}

\newtheorem{theorem}{Theorem}
\newtheorem{lemma}[theorem]{Lemma}

\newtheorem{remark}{Remark}
\usepackage{curves,amsmath,amssymb,color,fancybox,exscale,hyperref}
\usepackage{bm}

\usepackage{subfig}

\usepackage{graphicx}
\DeclareMathOperator{\ddiv}{div}

\newcommand{\jump}[1]{\ensuremath{
\lbrack\!\lbrack #1 \rbrack\!\rbrack
}}
\newcommand{\ajump}{\ensuremath{a_{\textrm{\tiny J}}}}

\newcommand{\Nedelec}{N\'{e}d\'{e}lec }

\numberwithin{equation}{section}

\graphicspath{{./}}

\begin{document}

\makeatletter
\def\@author#1{\g@addto@macro\elsauthors{\normalsize%
    \def\baselinestretch{1}%
    \upshape\authorsep#1\unskip\textsuperscript{%
      \ifx\@fnmark\@empty\else\unskip\sep\@fnmark\let\sep=,\fi
      \ifx\@corref\@empty\else\unskip\sep\@corref\let\sep=,\fi
    }%
    \def\authorsep{\unskip,\space}%
    \global\let\@fnmark\@empty
    \global\let\@corref\@empty  
    \global\let\sep\@empty}%
  \@eadauthor={#1}
}
\makeatother

\begin{frontmatter}

  \title{A Nonconforming Finite Element Method for the Biot's
    Consolidation Model in Poroelasticity\tnoteref{ztitlenote}}
  \tnotetext[ztitlenote]{The research of F.~J.~Gaspar and C.~Rodrigo
    is supported in part by the Spanish project FEDER/MCYT
    MTM2013-40842-P and the DGA (Grupo consolidado PDIE). The work of
    L.~T.~Zikatanov is supported in part by the National Science
    Foundation under contracts DMS-1418843 and DMS-1522615.}

   \author{Xiaozhe Hu\corref{cor1}}
   \address{Department of Mathematics, Tufts University,
 Medford, Massachusetts  02155, USA\, {Email:}xiaozhe.hu@tufts.edu}

   \cortext[cor1]{Corresponding Author}

   \author{Carmen Rodrigo}
 \address{Departamento de Matem\'{a}tica Aplicada,
 Universidad de Zaragoza, 50009 Zaragoza, Spain\,
 {Email:} carmenr@unizar.es}

 \author{Francisco J. Gaspar}
 \address{Departamento de Matem\'{a}tica Aplicada,
  Universidad de Zaragoza, 50009 Zaragoza, Spain\,
 {Email: } fjgaspar@unizar.es}

  \author{Ludmil T. Zikatanov}
  \address{Department of Mathematics, Penn State,
University Park, Pennsylvania, 16802, USA\,
{Email:} ludmil@psu.edu\\
Institute for Mathematics and
  Informatics, Bulgarian Academy of Sciences, Sofia, Bulgaria}

  \begin{abstract}
  A stable finite element scheme that avoids pressure oscillations for
  a three-field Biot's model in poroelasticity is considered. The
  involved variables are the displacements, fluid flux (Darcy velocity), and the pore
  pressure, and they are discretized by using the lowest possible
  approximation order: Crouzeix-Raviart finite elements for the
  displacements, lowest order Raviart-Thomas-\Nedelec elements for the Darcy velocity, and piecewise constant approximation for the pressure.  Mass-lumping technique is introduced for the Raviart-Thomas-\Nedelec elements in order to eliminate the Darcy velocity and, therefore, reduce the computational cost.
  We show convergence of the discrete scheme which is implicit in time
  and use these types of elements in space with and without mass-lumping.  Finally, numerical
  experiments illustrate the convergence of the method and show its
  effectiveness to avoid spurious pressure oscillations when mass lumping for the Raviart-Thomas-\Nedelec elements is used.
\end{abstract}


\end{frontmatter}

\section{Introduction: Biot's Model and Three Field Formulation} \label{sec:intro}

Poroelasticity theory mathematically describes the interaction between
the deformation of an elastic porous material and the fluid flow
inside of it. A pioneer in the mathematical modeling of such coupling
is Terzaghi~\cite{terzaghi} with his one-dimensional model.  Later, Biot~\cite{Biot1, Biot2} developed a three-dimensional
mathematical model, used to date, for quantitative and qualitative
study of poroelastic phenomena. Nowadays, the analysis and numerical
simulation of Biot's model become increasingly popular due to the wide range
of applications in medicine, biomechanics, petroleum engineering, food
processing, and other fields of science and
engineering.  

One of the main challenges in the numerical simulations based on the
Biot's models is the numerical instabilities in the approximation of the pressure variable (see~\cite{Gaspar2003,
  Ferronato2010, Langtangen2012, Favino2013, WheelerPhillip}). Such
instabilities occurs when materials have low permeability and/or a
small time step is used at the beginning of the consolidation
process. Different explanations have been provided in the literature
about the nature of these instabilities.  Usually, they are attributed
to violation of the \emph{inf-sup} condition for the Stokes problem,
or the lack of monotonicity in the discrete schemes.  Various
numerical methods to avoid this nonphysical behavior have been
analyzed. In several papers Murad, Loula, and
Thomee~\cite{MuradLoula92, MuradLoula94, MuradLoulaThome} studied the
case of stable discretizations (satisfying the inf-sup condition) for
the classical two-field formulation based on displacement and pressure
variables. As shown in~\cite{Aguilar2008, RGHZ2016}, however, the
inf-sup condition is not a sufficient condition, and for small
permeabilities the approximation to the pressure exhibits oscillations
and is numerically unstable.
A possible remedy is to add to the flow equation in the Biot's model
a time-dependent stabilization term, leading to oscillation free
approximations of the pressure. Such stabilized discretizations, based
on the MINI element~\cite{1984ArnoldD_BrezziF_FortinM-aa} as well as the P1-P1 element, are
proposed and analyzed in detail in~\cite{RGHZ2016}.
Other numerical schemes, such as least squares mixed finite element
methods, are proposed in~\cite{korsawe} and~\cite{Tchonkova} for a
four-field formulation (displacement, stress, fluid flux and
pressure). Different combinations of continuous and discontinuous
Galerkin methods and mixed finite element methods for a three-field
formulation are studied in~\cite{phillips1, phillips2, phillips3}.
Recently, conforming linear finite elements with stabilization for the three
field problem are proposed and analyzed
in~\cite{2015BergerL_BordasR_KayD_TavenerS-aa}.

Throughout this paper, we restrict our study to the quasi-static Biot's
model for soil consolidation. We assume the porous medium to be
linearly elastic, homogeneous, isotropic and saturated by an
incompressible Newtonian fluid.  According to Biot's
theory~\cite{Biot1}, then, the consolidation process must satisfy the
following system of partial differential equations:
\begin{eqnarray}
\mbox{\rm equilibrium equation:} & & -{\rm div} \, {\boldsymbol
\sigma}' +
\alpha \nabla \, p = {\bm g}, \quad {\rm in} \, \Omega, \label{eq11} \\
\mbox{\rm constitutive equation:} & & \bm{\sigma}' = 2\mu \varepsilon(\bm{u})  + \lambda\ddiv(\bm{u}) {\bm I}, \quad {\rm in} \, \Omega,
\label{eq12} \\
\mbox{\rm compatibility condition:} & & \varepsilon({\bm u}) = \frac{1}{2}(\nabla {\bm u} + \nabla
{\bm u}^t), \quad {\rm in} \, \Omega,
\label{eq13} \\
\mbox{\rm Darcy's law:} & & {\bm w} = - K \nabla
p, \quad {\rm in} \, \Omega,
\label{eq14} \\
\mbox{\rm continuity equation:} & & -
\alpha \ {\ddiv} \, \dot{\bm u}  -{\ddiv} \, {\bm w} = f, \quad {\rm in} \, \Omega,
\label{eq15}
\end{eqnarray}
where $\lambda$ and $\mu$ are the Lam\'{e} coefficients, $\alpha$ is
the Biot-Willis constant which we will assume equal to one without loss of generality, $K$ is the
hydraulic conductivity, given by the quotient between the permeability
of the porous medium $\kappa$ and the viscosity of the fluid $\eta$,
$\boldsymbol I$ is the identity tensor, ${\bm u}$ is the
displacement vector, $p$ is the pore pressure, ${\boldsymbol \sigma}'$
and $\boldsymbol \epsilon$ are the effective stress and strain tensors
for the porous medium and ${\bm w}$ is the percolation velocity of
the fluid relative to the soil. We denote the time derivative by a dot
over the letter. The right hand term ${\bm g}$ is the density of
applied body forces and the source term $f$ represents a forced fluid
extraction or injection process. Here, we consider a bounded open
subset $\Omega \subset {\mathbb R}^n,\; n \leq 3$ with regular
boundary $\Gamma$.

This mathematical model can also be written in terms of the
displacements of the solid matrix ${\bm u}$ and the pressure of
the fluid $p$. The displacement of the structure is described by
combining Hooke's law for elastic deformation with the momentum
balance equations, and the pressure of the fluid is described by
combining the fluid mass conservation with Darcy's law.
\begin{eqnarray}
&& -\ddiv\bm{\sigma}' + \nabla p = \bm{g},\qquad
\bm{\sigma}' = 2\mu \ \varepsilon(\bm{u})  + \lambda\ddiv(\bm{u}) \bm{I}, \label{two-field1}\\
&& -\ddiv \dot{\bm{u}} + \ddiv K \nabla p = f.\label{two-field2}
\end{eqnarray}
To complete the formulation of a well--posed problem we must add
appropriate boundary and  initial conditions. For instance,
\begin{equation}\label{bound-cond}
\begin{array}{ccccc}
  p = 0, & &  \quad \boldsymbol \sigma' \, {\bm n} &=& {\bm t}, \quad \hbox {on }\Gamma _t, \\
{\bm u} = {\bm 0}, & & \quad  \displaystyle K
(\nabla p)\cdot {\bm n} &=& 0, \quad \hbox {on } \Gamma _c,
\end{array}
\end{equation}
where ${\bm n}$ is the unit outward normal to the boundary and
$\Gamma_t \cup \Gamma_c = \Gamma$, with $\Gamma_t$ and $\Gamma_c$
disjoint subsets of $\Gamma$ with non null measure.
For the initial time, $t=0$, the following incompressibility condition is fulfilled
\begin{equation}\label{ini-cond}
      (\nabla \cdot {\bm u} )\, (\bm{x},0)=0, \,  \bm{x} \in\Omega.
\end{equation}

However, in many of the applications of the poroelasticity problem,
the flow of the fluid through the medium is of primary interest.
Although from the reduced displacement-pressure formulation the fluid
flux can be recovered, a natural approach is to introduce this value
as an extra primary variable instead. In this work, we are interested
in this three-field formulation of the problem. The extra unknown can
be seen as a disadvantage against the two-field formulation, regarding
the computational cost, but there are reasons to prefer this approach.
For example, the calculation of the fluid flux in post-processing is
avoided in the way that the order of accuracy in its computation is
higher and also the mass conservation for the fluid phase is ensured
by using continuous elements for the fluid flux variable. Therefore,
the governing equations of the Biot's model, with the displacement
${\bm u}$, Darcy velocity ${\bm w}$ and pressure $p$ as primary
variables are the following
\begin{eqnarray} && -\ddiv\bm{\sigma}' +
  \nabla p = \bm{g},\qquad
                                                \bm{\sigma}' = 2\mu \varepsilon(\bm{u})  + \lambda\ddiv(\bm{u}) \bm{I}, \label{three-field1}\\
                                             && {\bm w} + K \nabla p = 0, \label{three-field2}\\
                                             && -\ddiv \dot{\bm{u}} -
                                                \ddiv {\bm w} =
                                                f.\label{three-field3}
                            \end{eqnarray}

 Then, we can introduce the variational formulation for the three-field formulation of the Biot's model as follows: Find $(\bm u,\bm w, p)\in \bm V \times \bm W \times Q$ such that
\begin{eqnarray}
  && a(\bm{u},\bm{v}) - (p,\ddiv \bm{v}) = (\bm g,\bm{v}),
     \quad \forall \  \bm{v}\in \bm V, \label{variational1}\\
  && (K^{-1}\bm{w},\bm{r}) - (p,\ddiv \bm{r}) = 0, \quad \forall \ \bm{r}\in \bm W,\label{variational2}\\
  && -(\ddiv \dot{\bm{u}},q) - (\ddiv \bm{w},q)   = (f,q), \quad \forall \ q \in Q,\label{variational3}
\end{eqnarray}
where the considered functional spaces are
\begin{eqnarray*}
&&{\bm V} = \{{\bm u}\in {\bm H}^1(\Omega) \ |  \ {\bm u}|_{\Gamma _c} = {\bm 0} \},\\
&&{\bm W} = \{{\bm w} \in \bm{H}(\ddiv,\Omega) \ | \  ({\bm w}\cdot {\bm n})|_{\Gamma _c} = 0\}, \\
&&Q = L^2(\Omega),
\end{eqnarray*}
and the bilinear form $a(\bm{u},\bm{v})$ is given as the following,
\begin{equation}\label{bilinear}
a(\bm{u},\bm{v}) =
2\mu\int_{\Omega}\varepsilon(\bm{u}):\varepsilon(\bm{v}) +
\lambda\int_{\Omega} \ddiv\bm{u}\ddiv\bm{v},
\end{equation}
which corresponds to the elasticity part, and is a continuous bilinear
form.  Results on well-posedness of the continuous problem were
established by Showalter~\cite{showalter}, and, for the three field
formulation, Lipnikov~\cite{lipnikov_phd}.  In this work, we consider
a discretization of the three-field formulation with the displacement,
the Darcy velocity and the pressure as variables.  More precisely, we
use a nonconforming finite element space for the displacements
$\bm{u}$ and conforming finite element spaces for both the Darcy
velocity $\bm{w}$ and the pressure $p$.  As a time stepping technique,
we use the backward Euler method, and, we show that the resulting
fully discrete scheme has an optimal convergence order.  A similar
discretization was considered in \cite{Yi.S.2013a} for the 2D case on
rectangular grids.  Our scheme here works for both 2D and 3D cases on
simplicial meshes and has the potential to be extended to more general
meshes using the elements developed in
\cite{Di-Pietro.D;Lemaire.S2015a} and \cite{Kuznetsov.Y;Repin.S2003}.
We also show, albeit only numerically, that this nonconforming
three-field scheme produces oscillation-free numerical approximations
when mass-lumping is used in the Raviart-Thomas discretization for the
Darcy velocity. A four-field (the stress tensor, the
  fluid flux, the solid displacement, and the pore pressure)
  discretization has also been proposed and analyzed in
  \cite{Yi.S.2014a}. In \cite{Lee.J2015a}, an improved a priori error
  analysis for the four-field formulation has been discussed, in which
  the error estimates of all the unknowns are robust for material
  parameters.  We comment that our scheme with mass-lumping could be
  generalized to the four-field formulation as well.  


  The rest of the paper is organized as follows.  In Section
  \ref{sec:NC-FEM}, we introduce our nonconforming spatial
  semi-discrete scheme. The fully discrete scheme is discussed in
  Section \ref{sec:fully}, where we show the well-posedness of the
  discrete problem and derive error estimates for the nonconforming
  finite element approximations.  Section \ref{sec:numerics} is
  devoted to the numerical study of the convergence and tests the
  monotone behavior of the scheme and conclusions are drawn in
  Section \ref{sec:conclusion}.


%

\section{Nonconforming Discretization}\label{sec:NC-FEM}



In this section, we consider spatial semi-discretization using a
nonconforming finite element method.  We cover $\Omega$ by simplices
(triangles in 2D and tetrahedra in 3D) and have the following
finite element discretization corresponding to the three-field
formulation~\eqref{variational1}-\eqref{variational3}: Find
$(\bm u_h,\bm w_h, p_h)\in \bm V_h\times \bm W_h\times Q_h$ such that
\begin{eqnarray}
&& a_h(\bm{u}_h,\bm{v}_h) - (p_h,\ddiv \bm{v}_h) = (\bm{g}(t),\bm{v}_h),
\quad \forall \
\bm{v}_h\in \bm V_h,
\label{eqn:u_h-semi}
\\
&& (K^{-1}\bm{w}_h,\bm{r}_h)_h - (p_h,\ddiv \bm{r}_h) = 0, \quad \forall  \ \bm{r}_h \in \bm W_h
\label{eqn:w_h-semi}
\\
&& -(\ddiv \dot{\bm{u}}_h,q_h) - (\ddiv \bm{w}_h,q_h)   = (f(t),q_h),
\quad \forall \  q_h\in Q_h.
\label{eqn:p_h-semi}
\end{eqnarray}
Here, $\bm V_h$ is the Crouzeix-Raviart finite element
space~\cite{1973CrouzeixM_RaviartP-aa}, $\bm W_h$ the lowest order
Raviart-Thomas-\Nedelec
space~\cite{1977RaviartP_ThomasJ-aa,1986NedelecJ-aa,1980NedelecJ-aa},
and $Q_h$ is the space of piecewise constant functions (with respect to
the triangulation $\mathcal{T}_h$). Further details are discussed later in this section. 

\subsection{Interfaces, Normal Vectors and Jumps of Traces of Functions}
Let us first introduce some notation.  We denote the set of faces
(interfaces) in the triangulation $\mathcal{T}_h$ by $\mathcal{E}$ and
introduce the set of boundary faces $\mathcal{E}^{\partial}$, and the
set of interior faces $\mathcal{E}^o$. We have
$\mathcal{E} =\mathcal{E}^o\cup\mathcal{E}^\partial$.

Let us fix $e\in \mathcal{E}^o$ and let $T\in \mathcal{T}_h$ be such
that $e\in \partial T$. We set $\bm{n}_{e,T}$ to be the unit outward
(with respect to $T$) normal vector to $e$.  In addition, with every
face $e\in \mathcal{E}^o$, we also associate a unit vector $\bm{n}_e$
which is orthogonal to the $(d-1)$ dimensional affine variety (line in
2D and plane in 3D) containing the face. For the boundary faces, we
always set $\bm{n}_e=\bm{n}_{e,T}$, where $T$ is the \emph{unique}
element for which we have $e\in \partial T$.  In our setting, for the
interior faces, the particular direction of $\bm{n}_e$ does not really
matter, although it is important that this direction is fixed for
every face. Thus, for $e\in \mathcal{E}$, we define $T^{+}(e)$ and
$T^{-}(e)$ as follows:
\begin{equation*}
T^{\pm}(e):= \{T\in \mathcal{T}_h\  \mbox{such that} \ e\in \partial T, \
\mbox{and} \ (\bm{n}_{e}\cdot\bm{n}_{e,T})=\pm 1\}.
\end{equation*}
It is immediate to see that both sets defined above contain \emph{no
  more than} one element, that is: for every face we have exactly one
$T^{+}(e)$ and for the interior faces we also have exactly one
$T^{-}(e)$. For the boundary faces we only have $T^{+}(e)$. In the
following, we write $T^\pm$ instead of $T^\pm(e)$, when
this does not cause confusion and ambiguity.

Next, for a given function $u$ (vector or scalar valued) its jump
across an interior face $e\in \mathcal{E}^o$ is denoted by
$\jump{u}_e$, and defined as
\[
\jump{u}_e(x) = u_{T^+(e)}(x) - u_{T^-(e)}(x), \quad x\in e.
\]

\subsection{Finite Element Spaces}\label{FEM_method}
We now give the definitions of the finite element spaces
used in the semi-discretization~\eqref{eqn:u_h-semi}--\eqref{eqn:p_h-semi}.

\subsubsection{Nonconforming Crouzeix-Raviart Space}
The Crouzeix-Raviart space $\bm{V}_h$ consists of vector valued
functions which are linear on every element $T\in \mathcal{T}_h$ and
satisfy the following continuity conditions
\begin{eqnarray*}
&&\bm V_h = \left\{\bm v_h \in {\bm{L}}^2(\Omega) \;\bigg| \int_e\jump{\bm v_h}_e =0,\;\;
   \mbox{for all}\; e\in \mathcal{E}^o\right\}.
\end{eqnarray*}
Equivalently, all functions from $\bm V_h$ are continuous at the
barycenters of the faces in $\mathcal{E}^o$. For the boundary faces, the
elements of $\bm V_h$ are zero in the barycenters of any face on the
Dirichlet boundary.

\subsubsection{Raviart-Thomas-\Nedelec Space}
We now consider the standard lowest order Raviart-Thomas-\Nedelec space
$\bm{W}_h$.  Recall that every element $\bm{v}_h\in \bm{W}_h$ can be
written as
\begin{equation}\label{eq:RT}
\bm{v}_h = \sum_{e\in \mathcal{E}_h} e(\bm{v}_h )\bm{\psi}_e(\bm{x}),
\end{equation}
Here, $e(\cdot)$ denotes the functional (as known as the degree of
freedom) associated with the face $e\in \mathcal{E}$ and its action on
a function $\bm v$ for which $\bm v \cdot \bm n_e$ is in $L^2(e)$ is
defined as
\[
e(\bm{v}) = \int_e\bm v\cdot\bm n_e.
\]
To define $\bm{W}_h$, we only need to define the basis functions
$\bm{\psi}_e$, for $e\in \mathcal{E}$, dual to the degrees of freedom
$e(\cdot)$.  If $e$ is the face opposite to the vertex $P_e$ of the
triangle/tetrahedron $T$, then
\begin{equation}\label{eq:RT-basis}
\bm{\psi}_e\big|_T =
\frac{(\bm n_e\cdot \bm n_{e,T})}{d|T|}(\bm{x}-\bm{x}_{P_e})=
\pm \frac{1}{d|T|}(\bm{x}-\bm{x}_{P_e}).
\end{equation}

We note that explicit formulae similar to~\eqref{eq:RT-basis} are
available also for the case of lowest order Raviart-Thomas-\Nedelec
elements on $d$-dimensional rectangular elements (parallelograms or
rectangular parallelepipeds). For any such element
$T \in \mathcal{T}_h$ with faces parallel to the coordinate planes
(axes) let $\bm{\psi}_{k}^{\pm}$ denote  the basis function
corresponding to the functional $F_k^{\pm}(\cdot)$.
Clearly, the outward normal vectors to the faces of such an element
are the $\pm\bm{e}_k$, $k=1,\ldots,d$, where $\bm{e}_k$ is the $k$-th
coordinate vector in $\mathbb{R}^d$. Let
$\bm{x}_{M,k}^{\pm} \in \mathbb{R}^d$ be the mass center of the face
$F_{k}^{\pm}$, $k=1,\ldots,d$. We then have
\begin{equation}\label{eq:RT-basis-cube}
\bm{\psi}_k^{\pm}(\bm{x}) = \frac{(\bm{x}-\bm{x}_{M,k}^{\mp})^T \bm{e}_k  }{|T|} \bm{e}_k .
\end{equation}
 From this formula we see that over the finite element $T$ the basis functions
 $ \bm{\psi}_k^{\pm}(\bm{x})$ are linear in $\bm{x}_k$ and
 constant in the remaining variables in $\mathbb{R}^d$.

\subsubsection{Piecewise Constant Space} For approximating the
pressure,  we use the piecewise constant space spanned by the
characteristic functions of the elements, i.e.
$Q_h = \operatorname{span}\{\chi_T\}_{T\in\mathcal{T}_h}$.

\subsection{Approximate Variational Formulation}
We first consider the bilinear form
$a_h(\cdot,\cdot): \bm V_h \times \bm V_h\mapsto \mathbb{R}$.  Before
we write out the details, we have to assume that $\Gamma_c$ is
non-empty.  If $\Gamma_c = \emptyset$, i.e., $\Gamma_t = \Gamma$ (the
pure traction problem), $a(\cdot,\cdot)$ is a positive semidefinite
form and the dimension of its null space equals the number of edges on
the boundary (for both 2D and 3D). Therefore, the Korn's inequality
fails.  Even if $\Gamma_c \neq \emptyset$, for some cases, Korn's
inequality may fail for the standard discretization by
Crouzeix-Raviart elements without additional stabilization.  In summary, 
if we simply take $a_h(\cdot,\cdot) = a(\cdot,\cdot)$
then it does not satisfy the discrete Korn inequality and, therefore,
$a_h(\cdot,\cdot)$ is not coercive.  Moreover, it is also possible
that Korn's inequality hold, but the constant will approach infinity
as the mesh size $h$ approaches zero.  In another words, if we use
$a_h(\cdot,\cdot) = a(\cdot,\cdot)$, the coercivity constant blows up
when $h$ approaches zero.  For discussions on nonconforming linear
elements for elasticity problems and discrete Korn's inequality, we
refer to \cite{FalkR-1991aa, 1990FalkR_MorleyM-aa} for more details.




One way to fix the potential problem is to add stabilization.  The following perturbation of the bilinear form which does satisfy the Korn's inequality was proposed by Hansbo and Larson~\cite{HansboP_LarsonM-2003aa}.
\begin{eqnarray*}
  a_h(\bm v, \bm w) = a(\bm v, \bm w) + \ajump(\bm v,\bm w),\quad \mbox{where}\quad
  \ajump(\bm v,\bm w) = 2\mu\gamma_1
\sum_{e\in \mathcal{E}} h_e^{-1}\int_e \jump{\bm v}_e\jump{\bm w}_e.
\end{eqnarray*}
Here, the constant $\gamma_1>0$ is a fixed real number away from $0$
(i.e. $\gamma_1=\frac12$ is an acceptable choice).  As shown in Hansbo
and Larson~\cite{HansboP_LarsonM-2003aa} the bilinear form
$a_h(\cdot,\cdot)$ is positive definite and the corresponding error is
of optimal (first) order in the corresponding energy norm. Moreover,
the resulting method is \emph{locking free} and we use such
$a_h(\cdot, \cdot)$ in our nonconforming scheme.

\begin{remark}
  In~\cite{HansboP_LarsonM-2003aa}, the jump term $\ajump(\cdot, \cdot)$
  includes all the edges, i.e., the stabilization needs to be done on
  both interior and boundary edges.  In \cite{2004BrennerS-aa}, it has
  been shown that the jump stabilization only needs to be added to the
  interior edges and boundary edges with Neumann boundary conditions
  and the discrete Korn's inequality still holds.  In fact, in
  \cite{2006MardalK_WintherR-aa}, it is suggested that only the normal
  component of the jumps on the edges is needed for the stabilization
  in order to satisfy the discrete Korn's equality.  
\end{remark}

We next consider the bilinear form in~\eqref{eqn:w_h-semi}, denoted by
$(\cdot,\cdot)_h$. The first choice for such a form is just taking the
usual $L^2(\Omega)$ inner product, i.e.
$(\bm{w}, \bm{r})_h = (\bm{w}, \bm{r}) = \int_{\Omega} \bm{w} \cdot
\bm{r}$.
This is a standard choice and leads to a mass matrix in the
Raviart-Thomas-\Nedelec element when we write out the matrix form.

The second choice, which is the bilinear form we use here, is based on mass lumping in the Raviart-Thomas space, i.e.,
\begin{equation}\label{lump}
(\bm r, \bm s)_h = \sum_T\sum_{e\subset \partial T} \omega_e \ e(\bm r)e(\bm s).
\end{equation}
We refer to \cite{2006BrezziF_FortinM_MariniL-aa} and
\cite{1996BarangerJ_MaitreJ_OudinF-aa} for details on determining the
weights $\omega_e$, which are $\omega_e = \displaystyle\frac{|e|d_{e}}{d}$ with $d_e$ being
the signed distance between the Voronoi vertices adjacent to the face
$e$.  Such weights, in the two-dimensional case, are chosen so that
\begin{equation}\label{def:lumping}
(\bm w, \bm r)_h = \int_{\Omega}\bm w\cdot \bm r, \quad \bm w, \bm r \in \bm W_h \ \text{and} \ \bm{w}, \bm{r} \ \text{are piecewise constants},
\end{equation}
which implies the equivalence between $(\bm{w}, \bm{r})_h$ and the
standard $L^2$ inner product $(\bm{w}, \bm{r})$.  The situation in 3D
is a little bit involved since \eqref{def:lumping} in general does not
hold. Nevertheless, in
  \cite{2006BrezziF_FortinM_MariniL-aa}, it has been shown that the
  mixed formulation for Poisson equation using the mass-lumping
  maintains the optimal convergence order, which is what we need for
  the convergence analysis of our scheme later in Section
  \ref{sec:fully}.  
\begin{remark}
  As shown in \cite{2006BrezziF_FortinM_MariniL-aa}, the mass-lumping
  technique is quite general and works for both two- and
  three-dimensional cases.  For the convergence analysis of the
  mass-lumping, they assume that the circumcenters are inside the
  simplex.  Such partition exists in general (see
  \cite{Brandts.J;Korotov.S;Krizek.M;Solc.J2009a}).  Moreover, they
  also pointed out that this assumption is not strictly necessary and
  can be relaxed.  When the mesh contains pairs of right triangles in
  2D and right tetrahedrons in 3D, $d_e$ degenerates to zero and so is
  the weight $\omega_e$.  However, we can remedy by combining the
  pressure unknowns on these pairs to just one pressure unknown.
\end{remark}

In practice, such lumped mass approximation results in a block
diagonal matrix and, therefore, we can eliminate the Darcy velocity
$\bm{w}$ and reduce the three-field formulation to two-field
formulation involving only displacement $\bm{u}$ and pressure $p$.  In
practice, such elimination reduces the size of the linear system that
needs to be solved at each time step and save computational
cost. In the literature, there have been other similar
  techniques for eliminating the Darcy velocity $\bm{w}$.  For
  example, numerical integration
  \cite{Micheletti.S;Sacco.R;Saleri.F2001a} and multipoint flux mixed
  formulation \cite{Wheeler.M;Xue.G;Yotov.I2012a}. In addition, for
Biot's model, as shown by numerical experiments in Section
\ref{sec:numerics}, the lumped mass approximation actually gives an
oscillation-free approximation while maintains the optimal error
estimates.



\section{Analysis of the Fully Discrete Scheme} \label{sec:fully}

In this section, we consider the fully discrete scheme
of \eqref{variational1}--\eqref{variational3} at
time $t_n = n \tau$, $n = 1,2, \ldots$ as
following: Find
$(\bm u_h^n,\bm w_h^n, p_h^n)\in \bm V_h\times \bm W_h\times Q_h$ such
that
\begin{eqnarray}
&& a_h(\bm{u}_h^n,\bm{v}_h) - (p_h^n,\ddiv \bm{v}_h) = (\bm{g}(t_n),\bm{v}_h),
\quad\forall \
\bm{v}_h \in \bm V_h,
\label{eqn:u_h^n}
\\
&& (K^{-1}\bm{w}_h^n,\bm{r}_h)_h - (p_h^n,\ddiv \bm{r}_h) = 0,
\quad \forall \ \bm{r}_h \ \in \bm W_h
\label{eqn:w_h^n}
\\
&& -(\ddiv \bar{\partial}_t \bm{u}_h^n,q_h) - (\ddiv \bm{w}_h^n,q_h)   = (f(t_n),q_h),
\quad\forall \ q_h\in Q_h,
\label{eqn:p_h^n}
\end{eqnarray}
where $\tau$ is the time step size and
$\bar{\partial}_t \bm{u}_h^n: = (\bm{u}_h^n - \bm{u}_h^{n-1})/\tau$.
For the initial data $\bm{u}_h^0$, we use the discrete counterpart of
\eqref{ini-cond}, i.e.,
\begin{equation}\label{def:u_h^0}
\ddiv \bm{u}_h^0 = 0.
\end{equation}
We will first consider the well-posedness of the linear system
\eqref{eqn:u_h^n}-\eqref{eqn:p_h^n} at each time step $t_n$ and then
derive the error estimates for the fully discrete scheme.

\subsection{Well-posedness}
We consider the following linear system derived from
\eqref{eqn:u_h^n}-\eqref{eqn:p_h^n}: Find
$(\bm u_h,\bm w_h, p_h)\in \bm V_h\times \bm W_h\times Q_h$ such that
\begin{eqnarray}
&& a_h(\bm{u}_h,\bm{v}_h) - (p_h,\ddiv \bm{v}_h) = (\bm{g},\bm{v}_h),
\quad\forall \
\bm{v}_h\in \bm V_h,
\label{eqn:u_h}
\\
&& \tau (K^{-1}\bm{w}_h,\bm{r}_h)_h - \tau (p_h,\ddiv \bm{r}_h) = 0,
\quad \forall \ \bm{r}_h \ \in \bm W_h
\label{eqn:w_h}
\\
&& - (\ddiv \bm{u}_h,q_h) -  \tau (\ddiv \bm{w}_h,q_h)   = (\tilde{f},q_h),
\quad \forall \ q_h\in Q_h,
\label{eqn:p_h}
\end{eqnarray}
Here, to simplify the presentation, we have omitted the superscript
$n$ because the results are independent of the time step. We have
denoted $\tilde{f} := \tau f(t_n) - \ddiv\bm{u}_h^{n-1}$, and we note
that the relations \eqref{eqn:w_h} and \eqref{eqn:p_h} are obtained by
multiplying \eqref{eqn:w_h^n} and \eqref{eqn:p_h^n} with the time step
size $\tau$.

We equip the space $\bm{V} \times \bm{W} \times Q$ with the following norm
\begin{equation}\label{def:norm}
\| (\bm{u}, \bm{w}, p) \|_{\tau} : = \left(  \| \bm{u} \|^2_1 + \tau \| \bm{w} \|^2 + \tau^2 \| \ddiv \bm{w} \|^2 + \| p \|^2  \right)^{1/2},
\end{equation}
where $\| \cdot \|_1$ and $\| \cdot \|$ denote the standard $H^1$ norm
and $L^2$ norm, respectively.
In the analysis we need the following composite bilinear form (including all variables):
\begin{align*}
B(\bm{u}_h, \bm{w}_h, p_h; \bm{v}_h, \bm{r}_h, q_h) & :=
a_h(\bm{u}_h, \bm{v}_h) - (p_h, \ddiv \bm{v}_h) + \tau (K^{-1}\bm{w}_h,\bm{r}_h)_h \\
& \quad - \tau (p_h,\ddiv \bm{r}_h) - (\ddiv \bm{u}_h,q_h) - \tau(\ddiv \bm{w}_h,q_h).
\end{align*}
Note that
\begin{equation} \label{divW-Q}
\ddiv \bm{W}_h \subseteq Q_h.
\end{equation}
Further, we note the following continuity, coercivity and stability (inf-sup) conditions
on the bilinear forms involved in the definition of $B(\cdot, \cdot, \cdot; \cdot, \cdot, \cdot)$:
\begin{align}
&a_h(\bm{u}_h, \bm{v}_h)  \leq C_{\bm{V}} \| \bm{u}_h \|_1 \| \bm{v}_h \|_1, \quad \forall \ \bm{u}_h, \bm{v}_h \in \bm{V}_h, \label{ine:continuity-a_h}\\
&a_h(\bm{u}_h, \bm{u}_h)  \geq \alpha_{\bm{V}} \| \bm{u}_h \|_1^2, \quad \forall \ \bm{u}_h \in \bm{V}_h, \label{ine:coercivity}  \\
&c_K \| \bm{w}_h \|^2  \leq (K^{-1} \bm{w}_h, \bm{w}_h)_h \leq C_K \| \bm{w}_h \|^2, \quad \forall \ \bm{w}_h \in \bm{W}_h,  \label{ine:mass-lump-equiv}\\
& \inf_{p_h \in Q_h} \sup_{\bm{u}_h \in \bm{V}_h} \frac{(\ddiv\bm{u}_h, p_h)}{\| \bm{u}_h \|_1 \| p_h \|} = \beta_{\bm{V}} > 0. \label{ine:inf-sup-V}
\end{align}
Under these conditions, which are satisfied by our choice of finite element spaces and discrete bilinear forms, we have the following theorem showing the solvability of the
linear system~\eqref{eqn:u_h}-\eqref{eqn:p_h}.
\begin{theorem}
If the conditions~\eqref{divW-Q}-\eqref{ine:inf-sup-V} hold, then the bilinear form $B(\cdot, \cdot, \cdot; \cdot, \cdot, \cdot)$ satisfies the following inf-sup condition,
\begin{equation}\label{infsup-B}
\sup_{(\bm{v}_h, \bm{r}_h, q_h) \in \bm V_h\times \bm W_h\times Q_h}  \frac{B(\bm{u}_h, \bm{w}_h, p_h; \bm{v}_h, \bm{r}_h, q_h)}{\|  \left( \bm{v}_h,  \bm{r}_h,  q_h  \right) \|_{\tau}} \geq \gamma \| \left( \bm{u}_h,  \bm{w}_h, p_h \right) \|_{\tau}
\end{equation}
with a constant $\gamma > 0$ independent of mesh size $h$ and
time step size $\tau$.  Moreover, the three field formulation
\eqref{eqn:u_h}-\eqref{eqn:p_h} is well-posed.
\end{theorem}

\begin{proof}
  According to the inf-sup condition \eqref{ine:inf-sup-V}, we have
  for $p_h$, there exists $\bm{h}_h \in \bm{V}_h$, such that
\begin{equation}
(\ddiv \bm{h}_h, p_h) \geq \beta_{\bm{V}} \| p_h \|^2, \quad \| p_h \|= \| \bm{h}_h \|_1.
\end{equation}
Note that, according to the condition~\eqref{divW-Q},
$\ddiv \bm{w}_h \in Q_h$. Let
$\bm{v}_h = \bm{u}_h - \theta_1 \bm{h}_h$,
$\bm{r}_h = \bm{w}_h$,
$q_h = -(p_h + \theta_2 \tau \ddiv \bm{w}_h)$,
then we have,
\begin{align*}
  B(\bm{u}_h, \bm{w}_h, p_h; \bm{v}_h, \bm{r}_h, q_h)  & =
a_h(\bm{u}_h, \bm{u}_h - \theta_1 \bm{h}_h)
- (p_h, \ddiv (\bm{u}_h - \theta_1 \bm{h}_h) ) +  \tau (K^{-1} \bm{w}_h, \bm{w}_h)_h - \tau (p_h, \ddiv \bm{w}_h)
  \\ &\quad \ - (\ddiv \bm{u}_h, -p_h -\theta_2 \tau \ddiv \bm{w}_h ) - \tau (\ddiv \bm{w}_h,
-p_h - \theta_2 \tau \ddiv \bm{w}_h)
  \\ &  = \| \bm{u} _h \|_{a_h}^2 - \theta_1 a_h(\bm{u}_h, \bm{h}_h) + \theta_1 (p_h, \ddiv \bm{h}_h) + \tau ( K^{-1} \bm{w}_h, \bm{w}_h )_h
  \\ & \quad +  \theta_2 (\ddiv \bm{u}_h, \tau \ddiv \bm{w}_h) + \theta_2 \tau^2  \| \ddiv \bm{w}_h \|^2
  \\ & \geq \| \bm{u}_h \|_{a_h}^2 - \frac{\theta_1 \epsilon_1}{2} \| \bm{u}_h \|_{a_h}^2 - \frac{\theta_1}{2 \epsilon_1} \| \bm{h}_h \|_{a_h}^2 + \theta_1 \beta_{\bm{V}} \| p_h \|^2 +  c_K \,\tau \| \bm{w}_h \|^2
  \\ & \quad - \frac{\theta_2 \epsilon_2}{2} \| \ddiv \bm{u}_h \|^2 - \frac{\theta_2}{2 \epsilon_2} \tau^2 \| \ddiv \bm{w}_h \|^2 + \theta_2  \tau^2 \| \ddiv \bm{w}_h \|^2
  \\ &  \geq \left( 1 - \frac{\theta_1 \epsilon_1}{2}  \right) \| \bm{u}_h \|_{a_h}^2 - \frac{\theta_1 C_{\bm V}}{2 \epsilon_1}\| p_h \|^2 + \theta_1 \beta_{\bm{V}} \| p_h \|^2 + c_K \,  \tau \| \bm{w}_h \|^2
  \\ &  \quad- \theta_2 \epsilon_2 \| \bm{u}_h \|_1^2 - \frac{\theta_2}{2 \epsilon_2} \tau^2 \| \ddiv \bm{w}_h \|^2 + \theta_2 \tau^2 \| \ddiv \bm{w}_h \|^2
  \\ & \geq \left[ \alpha_{\bm{V}} \left( 1 - \frac{\theta_1 \epsilon_1}{2} \right)  - \theta_2 \epsilon_2\right] \| \bm{u}_h \|_1^2 +  c_K \,\tau \| \bm{w}_h \|^2 \\
                                                       & \quad + \theta_2 \left( 1 - \frac{1}{2\epsilon_2 }  \right) \tau^2 \| \ddiv \bm{w}_h \|^2  + \left[  \theta_1 \left( \beta_{\bm{V}} - \frac{C_{\bm V}}{2 \epsilon_1} \right) \right] \| p_h \|^2.
\end{align*}
Choose $\theta_1 = \frac{\beta_{\bm{V}}}{C_{\bm{V}}}$, $\epsilon_1 = \frac{C_{\bm{V}}}{\beta_{\bm{V}}}$, $\theta_2 = \frac{\alpha_{\bm{V}}}{3}$, and $\epsilon_2 = 1$, we have
\begin{align*}
B(\bm{u}_h, \bm{w}_h, p_h; \bm{v}_h, \bm{r}_h,q_h) & \geq \frac{\alpha_{\bm{V}}}{6} \| \bm{u}_h \|_1^2 + c_K \,\tau \| \bm{w}_h \|^2  + \frac{\alpha_{\bm{V}}}{6}  \tau^2 \| \ddiv \bm{w}_h \|^2 + \frac{\beta_{\bm{V}}^2}{2 C_{\bm{V}}} \| p_h \|^2
	\\ & \geq C_1 \| \left(  \bm{u}_h, \bm{w}_h, p_h \right) \|_{\tau}^2,
\end{align*}
where $\displaystyle C_1 = \min \left\{  \frac{\alpha_{\bm{V}}}{6},  c_K ,  \frac{\beta_{\bm{V}}^2}{2 C_{\bm{V}}}\right\} $.

On the other hand, we have
\begin{align*}
\| \left( \bm{v}_h, \bm{r}_h, q_h \right) \|_{\tau}^2 & = \| \bm{u}_h - \theta_1 \bm{h}_h \|_1^2 + \tau \| \bm{w}_h \|^2 + \tau^2 \| \ddiv \bm{w}_h\|^2 + \|-p_h - \theta_2 \tau \ddiv \bm{w}_h \|^2
	\\ & \leq 2 \| \bm{u}_h \|_1^2 + 2 \theta_1^2 \| \bm{h}_h \|_1^2 +  \tau \| \bm{w}_h \|^2 + \tau^2 \| \ddiv \bm{w}_h\|^2 + 2\| p_h \|^2 + 2 \theta_2^2 \tau^2 \| \ddiv \bm{w}_h \|^2
	\\ & = 2 \| \bm{u}_h \|_1^2  +  \tau \| \bm{w}_h \|^2  +  \left( 1  +
\frac{2\alpha_{\bm{V}}^2}{9}  \right) \tau^2 \| \ddiv \bm{w}_h \|^2 +  \left( 2
\frac{\beta_{\bm{V}}^2}{C_{\bm{V}}^2} + 2  \right) \| p_h \|^2
	\\ & \leq C_2 \| \left(  \bm{u}_h, \bm{w}_h, p_h \right) \|_{\tau}^2,
\end{align*}
where $\displaystyle C_2:= \max \left\{ 2,  1 +  \frac{2 \alpha^2_{\bm{V}}}{9},  2 \frac{\beta_{\bm{V}}^2}{C_{\bm{V}}^2} + 2  \right \}  $.
Then \eqref{infsup-B} follows with $\gamma : = C_1 C_2^{\frac{1}{2}} $.

Moreover, it is easy to show that the bilinear form $B(\bm{u}_h, \bm{w}_h, p_h; \bm{v}_h, \bm{r}_h, q_h)$ is continuous, therefore, we can conclude that the three-field formulation is well-posed.
\end{proof}

\begin{remark}
  The continuity condition~\eqref{ine:continuity-a_h} follows from the
  definition of bilinear form and the corresponding norms. The coercivity
  condition~\eqref{ine:coercivity} follows from discrete Korn's
  equality, which hinges on the stabilization provided by the
   jump-jump term $\ajump(\cdot,\cdot)$. We refer to
  \cite{HansboP_LarsonM-2003aa} for details on this.  The
  condition~\eqref{ine:mass-lump-equiv} follows from the property of
  the lumped mass procedure and we refer to
  \cite{2006BrezziF_FortinM_MariniL-aa} for details.  The last
  condition \eqref{ine:inf-sup-V} is the standard inf-sup condition
  for the nonconforming finite element methods for solving the Stokes
  equation, see \cite{1973CrouzeixM_RaviartP-aa} for details.
\end{remark}

\subsection{Error Estimates}
To derive the error analysis of the fully discrete scheme
\eqref{eqn:u_h^n}-\eqref{eqn:p_h^n}, following the standard error
analysis of time-dependent problems in
Thom\'{e}e~\cite{2006ThomeeV-aa}, we first define the following
elliptic projections $\bar{\bm{u}}_h \in \bm{V}_h$,
$\bar{\bm{w}}_h \in \bm{W}_h$, and $\bar{p}_h \in Q_h$ for $t>0$ as
usual,
\begin{align}
& a_h(\bar{\bm{u}}_h, \bm{v}_h) - (\bar{p}_h, \ddiv \bm{v}_h) =
a_h(\bm{u}, \bm{v}_h) - (p, \ddiv \bm{v}_h),
\quad \forall \bm{v}_h \in \bm{V}_h, \label{eqn:elliptic-proj-u}\\
& (K^{-1} \bar{\bm{w}}_h, \bm{r}_h )_h - (\bar{p}_h, \ddiv \bm{r}_h) =
(K^{-1} \bm{w}, \bm{r}_h ) - (p, \ddiv \bm{r}_h),
\quad \forall \bm{r}_h \in \bm{W}_h, \label{eqn:elliptic-proj-w}\\
& (\ddiv \bar{\bm{w}}_h, q_h) =  (\ddiv \bm{w}, q_h), \quad \forall q_h \in Q_h. \label{eqn:elliptic-proj-p}
\end{align}
Note that the above elliptic projections are actually decoupled;
$\bar{\bm{w}}_h$ and $\bar{p}_h$ are defined by
\eqref{eqn:elliptic-proj-w} and \eqref{eqn:elliptic-proj-p} which is
the mixed formulation of the Poisson equation.  Therefore, the
existence and uniqueness of $\bar{\bm{w}}_h$ and $\bar{p}_h$ follow
directly from the standard results of mixed formulation of the Poisson
equation (for the mass-lumping case, we refer to
\cite{2006BrezziF_FortinM_MariniL-aa} for details).  After $\bar{p}_h$
is defined, $\bar{\bm{u}}_h$ can be determined by solving
\eqref{eqn:elliptic-proj-u} which is a linear elasticity problem, and
the existence and uniqueness of $\bar{\bm{u}}_h$ also follow from the
standard results of the linear elasticity problem.   Now we can split
the errors as follows
\begin{align*}
& \bm{u}(t_n) - \bm{u}_h^n = \left( \bm{u}(t_n) - \bar{\bm{u}}_h(t_n) \right) - \left(  \bm{u}_h^n - \bar{\bm{u}}_h(t_n)  \right) =: \rho_{\bm{u}}^n - e_{\bm{u}}^n, \\
&\bm{w}(t_n) - \bm{w}_h^n = \left( \bm{w}(t_n) - \bar{\bm{w}}_h(t_n) \right) - \left(  \bm{w}_h^n - \bar{\bm{w}}_h(t_n)  \right) =: \rho_{\bm{w}}^n - e_{\bm{w}}^n, \\
&p(t_n) - p_h^n = \left( p(t_n) - \bar{p}_h(t_n) \right) - \left(  p_h^n - \bar{p}_h(t_n)  \right) =: \rho_{p}^n - e_{p}^n.
\end{align*}
For the errors for the elliptic projections, we have, for $t>0$,
\begin{align}
& \| \rho_{\bm{u}} \|_{a_h} \leq c h \left( \| \bm{u} \|_2 + \| p \|_1 \right), \label{ine:rho_u} \\
& \| \rho_{\bm{w}} \| \leq c h \| \bm{w} \|_1, \label{ine:rho_w} \\
& \| \rho_p \| \leq c h \left(  \| p \|_1 + \| \bm{w} \|_1 \right). \label{ine:rho_p}
\end{align}
Note that \eqref{ine:rho_u} follows from the standard error analysis of linear
elasticity problems. \eqref{ine:rho_w} and \eqref{ine:rho_p} follow
from the error analysis of the mixed formulation of Poisson
problems. If the mass-lumping is applied, such error analysis can be
found in \cite{2006BrezziF_FortinM_MariniL-aa}.

We can similarly define the elliptic projection
$\overline{\partial_t \bm{u}}$, $\overline{\partial_t \bm{w}}$, and
$\overline{\partial_t p}$ of $\partial_t \bm{u}$, $\partial_t \bm{w}$,
and $\partial_t p$ respectively. And we have the estimates above also
for $\partial_t \rho_{\bm{u}}$, $\partial_t \rho_{\bm{w}}$, and
$\partial_t \rho_p$ as well, where on the right hand side of the
inequalities we have the norms of $\partial_t \bm{u}$,
$\partial_t \bm{w}$, and $\partial_t p$ instead of the norms of
$\bm{u}$, $\bm{w}$, and $p$ respectively.

We define the following norm on the finite element spaces:
\begin{equation*}
\| (\bm u, \bm w, p) \|_{\tau,h} := \left(  \| \bm u \|_{a_h}^2 + \tau \| \bm w \|_{K^{-1},h}^2 + \| p \|^2  \right)^{1/2},
\end{equation*}
where $\| \bm{w} \|_{K^{-1}, h}^2:= (K^{-1} \bm{w}, \bm{w} )_h$.

Now we need to estimate the errors $e_{\bm{u}}$, $e_{\bm{w}}$, and
$e_{p}$, and then the overall error estimates can be derived by the
triangular inequality.  Next lemma gives the error estimates of
$e_{\bm{u}}$, $e_{\bm{w}}$, and $e_{p}$.

\begin{lemma} \label{lem:error-e}
Let $R_{\bm{u}}^j := \partial_t \bm{u}(t_j) - \frac{\bar{\bm{u}}_h(t_j) - \bar{\bm{u}}_{h}(t_{j-1})}{\tau}$, we have
\begin{equation} \label{ine:error-e}
\|  (e_{\bm u}^n, e_{\bm w}^n, e_p^n) \|_{\tau, h} \leq c \left(  \| e_{\bm u}^0 \|_{a_h} +  \tau \sum_{j=1}^n \| R_{\bm u}^j \|_{a_h}   \right)
\end{equation}
\end{lemma}
\begin{proof}
  Choosing $\bm v = \bm v_h$ in \eqref{variational1},
  $\bm r = \bm r_h$ in \eqref{variational2}, and $q = q_h$ in
  \eqref{variational3}, and subtracting these three equations from
  \eqref{eqn:u_h^n}, \eqref{eqn:w_h^n}, and \eqref{eqn:p_h^n}, we have
\begin{align}
& a_h(e_{\bm u}^n, \bm v_h) - (e_{p}^n, \ddiv \bm v_h) = 0, \label{eqn:eu}\\
& (K^{-1} e_{\bm w}^n, \bm r_h)_h - (e_p^n, \ddiv \bm r_h) = 0,\label{eqn:ew} \\
& - (\ddiv \bar{\partial}_t e_{\bm u}^n, q_h) - (\ddiv e_{\bm w}^n, q_h) = - (\ddiv R_{\bm u}^n, q_h). \label{eqn:ep}
\end{align}
Choosing $\bm v_h = \bar{\partial}_t e_{\bm u}^n$,
$\bm r_h = e_{\bm w}^n$ and $q_h = -e_p^n$ in \eqref{eqn:eu},
\eqref{eqn:ew}, and \eqref{eqn:ep}, respectively, and adding them, we
have
\begin{align*}
\| e_{\bm u}^n \|_{a_h}^2 + \tau \| e_{\bm w}^n \|_{K^{-1},h}^2 & =
a_h(e_{\bm u}^n, e_{\bm u}^{n-1}) + \tau (\ddiv R_{\bm u}^n, e_p^n)  \leq \| e_{\bm u}^n \|_{a_h} \| e_{\bm u}^{n-1} \|_{a_h} + \tau \| \ddiv R_{\bm u}^n \| \| e_p^n \|.
\end{align*}
Thanks to the inf-sup conditions \eqref{ine:inf-sup-V} and
\eqref{eqn:eu}, we have
\begin{equation} \label{ine:infsup-ep}
\| e_p^n \| \leq c \sup_{0 \neq \bm v_h \in \bm V_h} \frac{(e_p^n, \ddiv \bm v_h)}{\| \bm v_h \|_{a_h}} = c \sup_{0 \neq \bm v_h \in \bm V_h} \frac{a_h(e_{\bm u}^n, \bm v_h)}{\| \bm v_h \|_{a_h}} = c \| e_{\bm u}^n \|_{a_h}.
\end{equation}
Therefore, we have
\begin{align} \label{ine:eu-ew}
\| e_{\bm u}^n \|_{a_h}^2 + \tau \| e_{\bm w}^n \|_{K^{-1},h}^2 \leq \| e_{\bm u}^n \|_{a_h} \left( \| e_{\bm u}^{n-1} \|_{a_h} + c \tau \| R_{\bm u}^n \|_{a_h}  \right).
\end{align}
This implies
\begin{align*}
\| e_{\bm u}^n \|_{a_h}  & \leq  \| e_{\bm u}^{n-1} \|_{a_h} + c \tau \| R_{\bm u}^n \|_{a_h}.
\end{align*}
By summing over all time steps, we have
\begin{align} \label{ine:eu-all}
\| e_{\bm u}^n \|_{a_h}  & \leq  \| e_{\bm u}^0 \|_{a_h} + c \tau  \sum_{j=1}^n \| R_{\bm u}^j \|_{a_h}.
\end{align}
Combining \eqref{ine:infsup-ep}, \eqref{ine:eu-ew}, and \eqref{ine:eu-all}, we have the estimate \eqref{ine:error-e}.
\end{proof}

Following the same procedures of Lemma 8 in \cite{RGHZ2016}, we have
\begin{equation}\label{ine:Ru}
\sum_{j=1}^n \| R_{\bm u}^j \|_{a_h} \leq c \left(  \int_0^{t_n} \| \partial_{tt} \bm u  \|_1 \mathrm{d}t + \frac{1}{\tau} \int_0^{t_n} \| \partial_t \rho_{\bm u} \|_1 \mathrm{d}t \right).
\end{equation}
Then we can derive the error estimates as shown in the following theorem.

\begin{theorem}\label{thm:error}
  Let $\bm u$, $\bm w$, and $p$ be the solutions of
  \eqref{variational1}-\eqref{variational3} and $\bm u_h^n$,
  $\bm w_h^n $, and $p_h^n$ be the solutions of
  \eqref{eqn:u_h^n}-\eqref{eqn:p_h^n}. If the following regularity assumptions hold,
\begin{align*}
&\bm u(t) \in L^{\infty}\left((0, T], \mathbf{H}_0^1(\Omega) \right) \cap L^{\infty}\left((0, T], \mathbf{H}^2(\Omega) \right), \\
& \partial_t \bm u \in L^{1}\left((0, T], \mathbf{H}^2(\Omega) \right), \ \partial_{tt} \bm u \in L^{1}\left((0, T], \mathbf{H}^1(\Omega) \right), \\
&  \bm w(t) \in L^{\infty}\left((0, T], H_0(\ddiv, \Omega) \right) \cap L^{\infty}\left((0, T], \mathbf{H}^1(\Omega) \right), \\
& p \in L^{\infty}\left((0, T], H^1(\Omega) \right), \ \partial_t p \in L^{1}\left((0, T], H^1(\Omega) \right),
\end{align*}
then we have the error estimates
\begin{align} \label{ine:error}
& \|( \bm u(t_n) - \bm u_h^n, \bm w(t_n) - \bm w_h^n, p(t_n) - p_h^n ) \|_{\tau, h}  \leq c \left\{ \| e_{\bm u}^0 \|_{a_h} +  \tau \int_{0}^{t_n} \| \partial_{tt} \bm u \|_1 \mathrm{d}t  \right. \nonumber  \\
&\qquad  \left. + h \left[  \| \bm u \|_2 + \tau^{1/2} \| \bm w\|_1 + \| \bm w\|_1 + \| p \|_1 + \int_0^{t_n} \left( \| \partial_t \bm u \|_2 + \| \partial_t p \|_1  \right) \mathrm{d}t \right] \right\}.
\end{align}
\end{theorem}

\begin{proof}
  The estimate \eqref{ine:error} follows directly from
  \eqref{ine:error-e}, \eqref{ine:Ru},
  \eqref{ine:rho_u}-\eqref{ine:rho_p}, and triangle inequality.
\end{proof}

\section{Numerical Tests} \label{sec:numerics}

In this section we consider two test cases verifying different aspects
of the questions and the analysis we have discussed earlier.  The
first numerical experiment uses analytical solution of a poroelastic
problem and confirms the accuracy of the discretization and the
results of error analysis presented in Section~\ref{sec:fully}. The
second test shows that the mass lumping technique, which can be viewed
as a stabilization, provides oscillation-free numerical solution for
the pressure field. Both numerical experiments take place in the unit
square as a computational domain, $\Omega = (0,1)\times(0,1)$; the
triangulation of $\Omega$ is obtained by partitioning $\Omega$ using
$n_x\times n_y$ rectangular grid, followed by splitting each
rectangle in two triangles by using one of its diagonals.

\subsection{Model Problem with Analytical Solution for the Convergence Study}

In this first numerical test we illustrate the theory described in
 Section~\ref{sec:fully}. We consider 
the poroelastic
problem~\eqref{three-field1}-\eqref{three-field3}, where
the source terms ${\mathbf g}$ and $f$ are chosen so that
the components of the exact solution, $\bm{u} = (u, v)^T$, and, $p$, are
\begin{eqnarray}
&u(x,y,t) = v(x,y,t) = e^{-t}\sin \pi x\, \sin \pi y,& \label{analytical_sol1}\\
&p(x,y,t) = e^{-t}(\cos \pi y + 1).& \label{analytical_sol2}
\end{eqnarray}
We prescribe homogeneous Dirichlet conditions for the
displacements $u$. Next, the whole boundary $\partial\Omega$
except its north edge is assumed impermeable, that is, $\nabla p \cdot {\mathbf n} = 0$
(equivalent to essential boundary conditions for the fluid flux
$\bm{w}$).  The material properties are:
Young modulus $E=1$ and the Poisson ratio $\nu = 0.2$.
and the permeability is assumed to be $\kappa = 1$.  The numerical
test provides estimates on the error between the exact solution
given in~\eqref{analytical_sol1}-\eqref{analytical_sol2} and the
numerical solution on progressively refined grids with $n_x=n_y$
ranging from $4$ to $64$. The time-steps ($\tau = T/n_t$) vary from
$1/4$ to $1/64$.  The errors in the displacements, measured in energy
norm, and, the errors in the pressure, measured in the $L^2$-norm, are
reported in Table~\ref{errors}.
\begin{table}[htb]
\centering
\begin{tabular}{||c|c|c|c|c|c||}
  \hline
  $n_x\times n_y \times n_t$ & $4\times 4 \times 4$ & $8\times 8 \times 8$ & $16\times 16 \times 16$ & $32\times 32 \times 32$ & $64\times 64 \times 64$ \\
  \hline
  $||{\mathbf u}-{\mathbf u}_h||_{a_h}$ & $0.2060$ & $0.1073$ & $0.0546$ & $0.0275$ & $0.0138$ \\
  \hline
  $||p-p_h||$ & $0.0476$ & $0.0194$ & $0.0092$ & $0.0045$ & $0.0023$ \\
  \hline
\end{tabular}
\caption{Energy norm of the displacements' error and $L^2$-norm of the pressure error for different spatial-temporal grids.}
\label{errors}
\end{table}
From the results reported in the Table~\ref{errors}, we observe first
order convergence, which is consistent with the error estimates obtained in
previous section.

\subsection{Poroelastic Problem on a Square Domain with a Uniform Load}

The second numerical experiment models an structure which drains on the
north (top) edge of the boundary. On this part of the boundary we also
apply a uniform unit load. More specifically we have,
$$
p = 0, \quad {\boldsymbol \sigma}\cdot {\mathbf n} = (0,-1)^t, \;
\hbox{on} \; \Gamma_1 = [0,1]\times{1},
$$
On the rest of the boundary we have impermeable boundary conditions for the pressure and
we also assume rigidity, namely, the rest of the boundary conditions are:
$$
\nabla p \cdot {\mathbf n} = 0, \quad {\mathbf u}={\mathbf 0}, \;
\hbox{on} \; \Gamma_2 = \partial \Omega \setminus \Gamma_1.
$$
For clarity, the prescribed boundary conditions are shown in
Figure~\ref{domain_test2}.
\begin{figure}[htb]
\begin{center}
\includegraphics*[width = 0.4\textwidth]{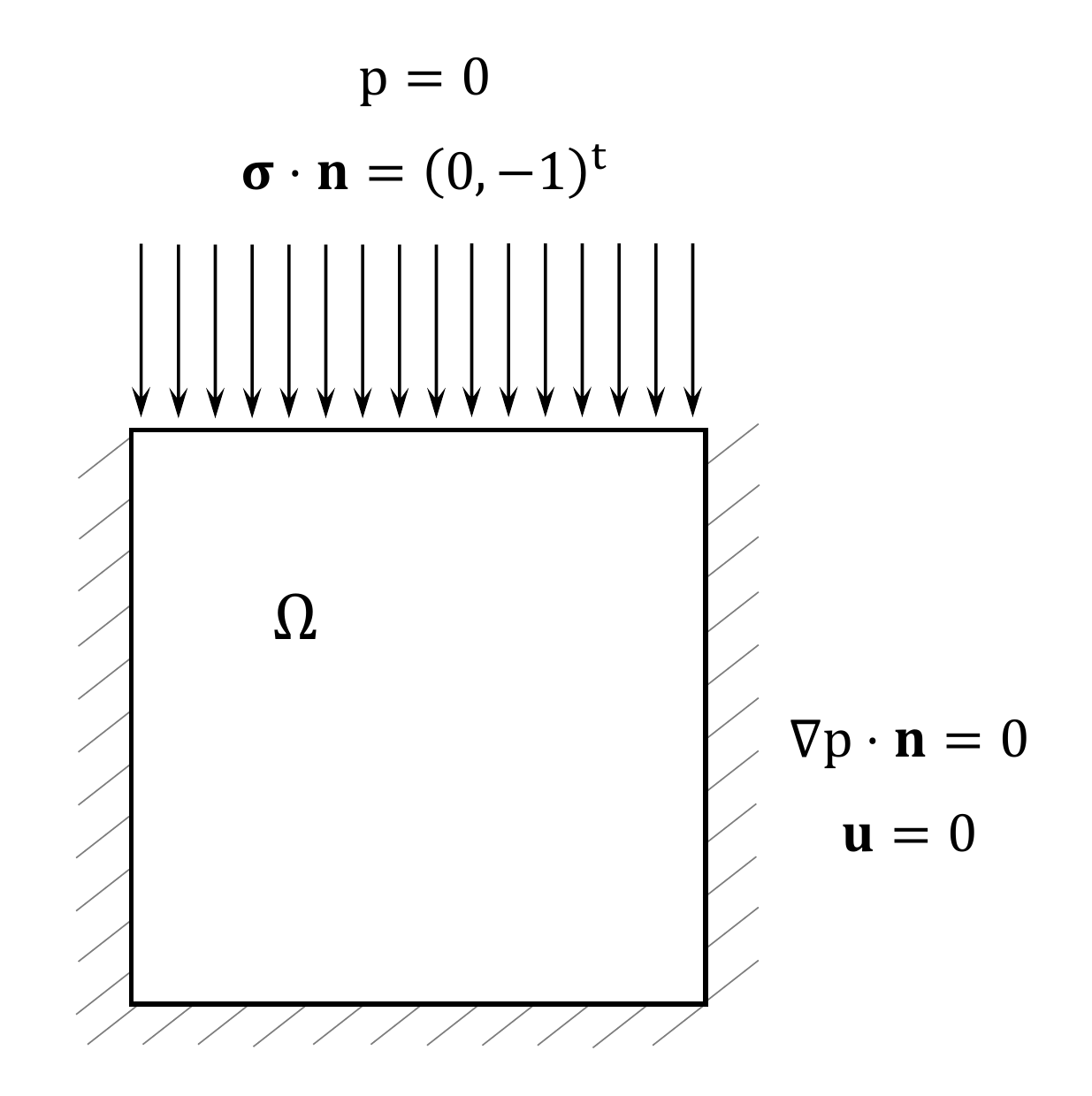}
\caption{Computational domain and boundary conditions corresponding to the second test problem.}
\label{domain_test2}
\end{center}
\end{figure}
\begin{figure}[htb]
\begin{center}
\begin{tabular}{cc}
\includegraphics*[width = 0.4\textwidth]{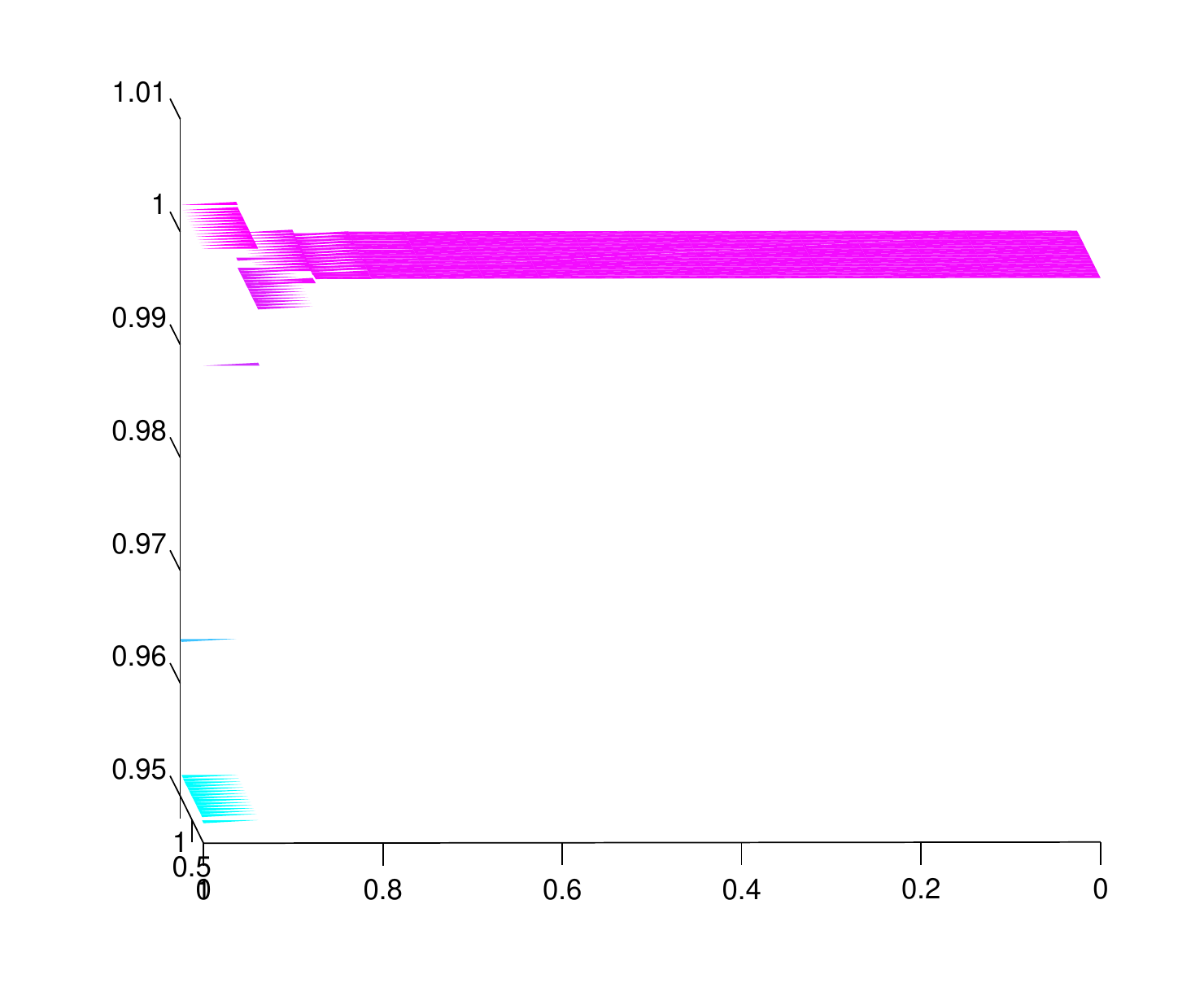}
&
\includegraphics*[width = 0.4\textwidth]{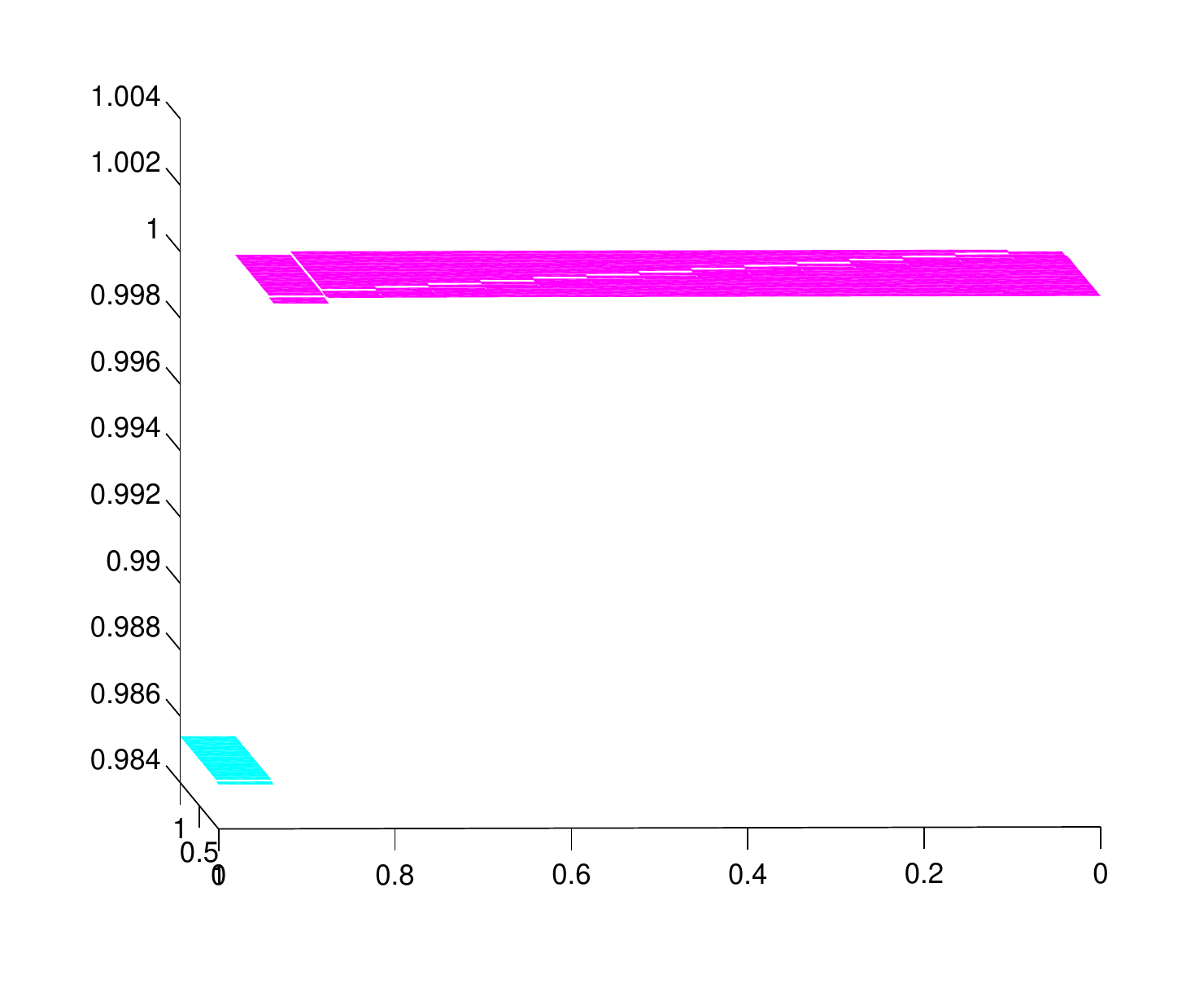}
\\
(a)
&
(b)
\end{tabular}
\caption{Numerical solution for the pressure to the second numerical experiment (a) without mass lumping and (b) with mass lumping.}
\label{pressure_experiment2}
\end{center}
\end{figure}

Here we aim to illustrate the stabilization effect of the mass-lumping
performed in the Raviart-Thomas-\Nedelec space (see~\eqref{lump}).  As
oscillations of the pressure usually occur when the material has low
permeability and for a short time interval, we set the final time
as $T=10^{-3}$ and perform only one time step. The value of the
permeability is $\kappa = 10^{-6}$ and the rest of the material
parameters (Lam\'e coefficients) are $\lambda = 12500$ and
$\mu = 8333$.

In Figure~\ref{pressure_experiment2}(a), we show the
approximation for the pressure field obtained without mass lumping.
We clearly observe small oscillations close to the boundary where
the load is applied. Introducing mass lumping in computing the
fluid flux completely removes these oscillations. This is also clearly
seen in Figure~\ref{pressure_experiment2}(b).

Another test is illustrated in Figure~\ref{pressure_experiment3}(a)-(b). We
show the numerical solutions for the same problem but with variable
permeability, i.e. $\kappa(x) = 10^{-3}$,
$x\in ((0,0.5]\times (0,0.5]) \cup ([0.5,1)\times [0.5,1))$ and
$\kappa=1$ in the rest of the domain. While the small oscillations in the solution shown in
Figure~\ref{pressure_experiment3}(a) are difficult to see, the lumped mass solution
shown in Figure~\ref{pressure_experiment3}(b) is clearly oscillation free.
\begin{figure}[htb]
\begin{center}
\begin{tabular}{cc}
\includegraphics*[width = 0.4\textwidth]{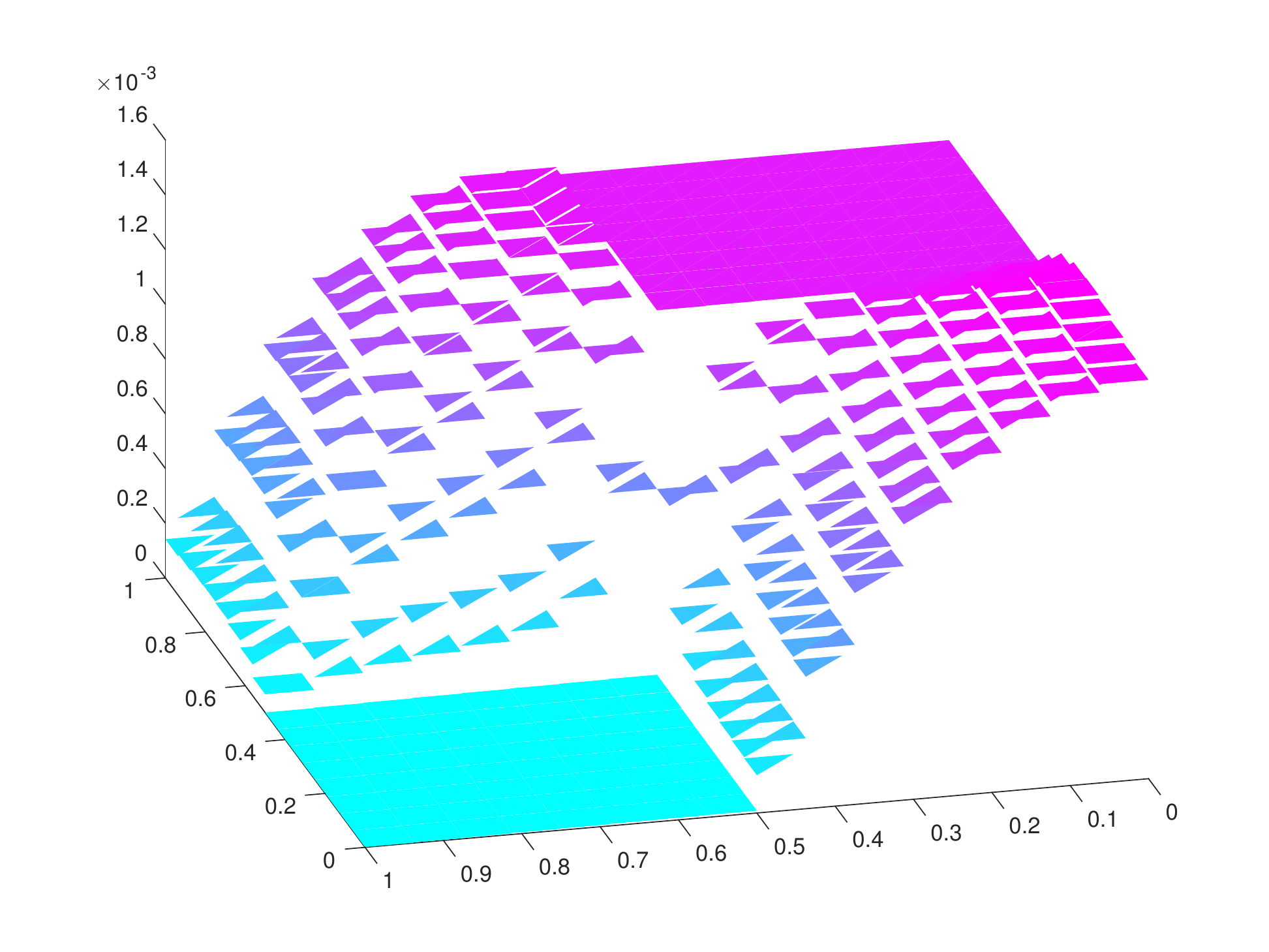}
&
\includegraphics*[width = 0.4\textwidth]{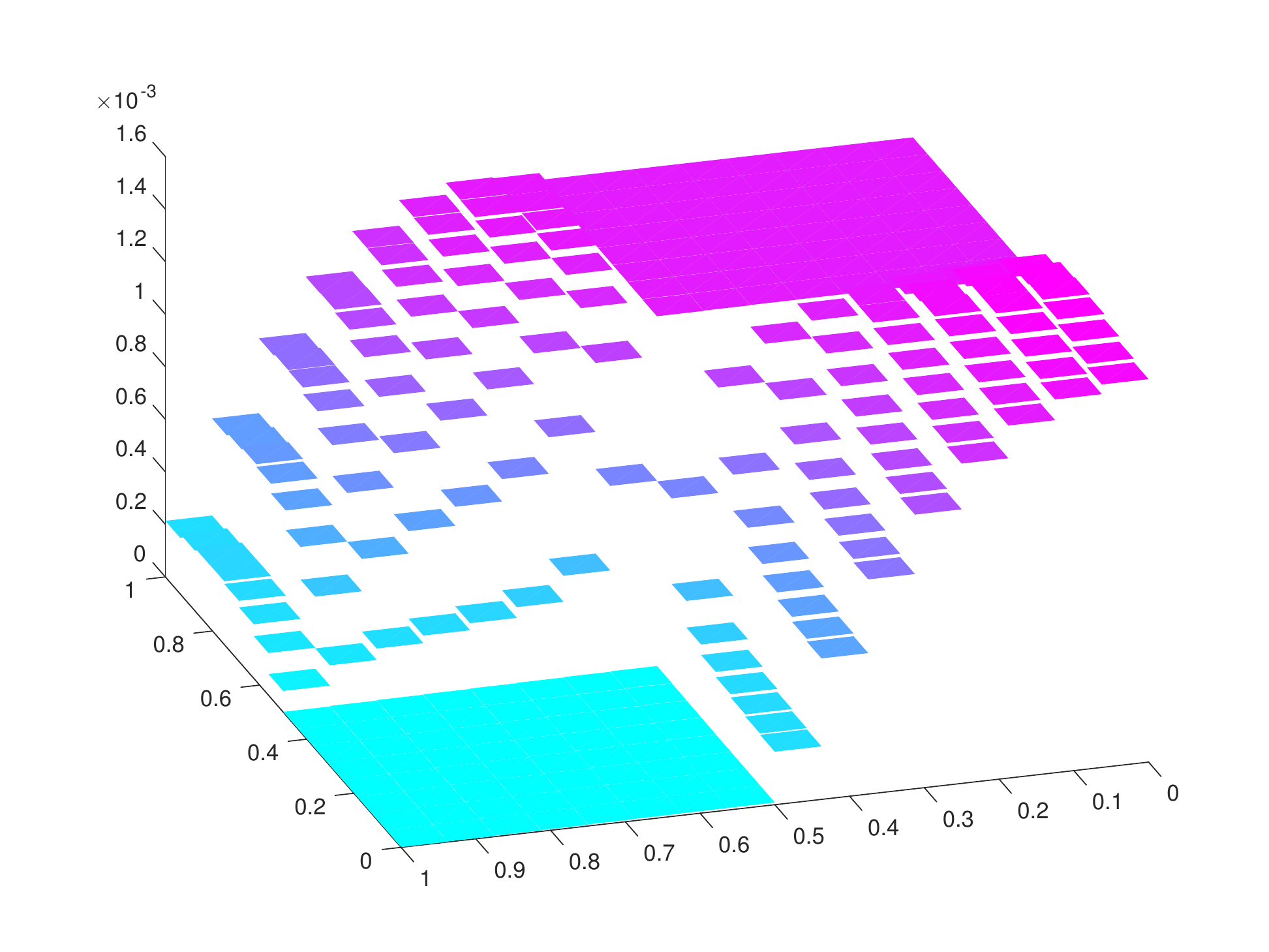}
\\
(a)
&
(b)
\end{tabular}
\caption{Numerical solution for the pressure for the example with variable permeability:  (a) without mass lumping and (b) with mass lumping.}
\label{pressure_experiment3}
\end{center}
\end{figure}

Finally, let us remark that, while illustrating that the mass-lumping
techniques remove the oscillations in the numerical solution, even in
this simple case there is no supporting theory showing that the
discretization we have analyzed is in fact monotone and will provide
oscillation free approximation to the pressure. This is a difficult
and interesting mathematical question which is still open.

\section{Conclusions} \label{sec:conclusion} In this paper, we have
proposed a nonconforming finite element method for the three-field
formulation for the Biot's model.  We use the lowest order finite
elements: piecewise constant for the pore pressure paired with the lowest order
Raviart-Thomas-\Nedelec elements for the Darcy's velocity and the
nonconforming Crouzeix-Raviart elements for the displacements.  The
time discretization is an implicit (backward) Euler method.  The
results on stability and error estimates, however, hold for other
implicit time stepping methods as well.  For the resulting fully
discrete scheme, we have shown uniform inf-sup condition for the
discrete problem. Further, based on standard decomposition of the
error for the time-dependent problem, we derived optimal order
error estimates in both space and time.  Finally, we presented
numerical tests confirming the theoretical estimates, and, in addition
showing that the mass lumping technique to eliminate the Darcy
velocity leads to an oscillation-free approximation of the pressure.

\section{Acknowledgements}
 Ludmil Zikatanov gratefully acknowledges the support for this work
 from the Department of Mathematics at Tufts University and the
 Applied Mathematics Department at University of Zaragoza.

\section*{\refname} 

\bibliographystyle{elsarticle-num}
\bibliography{bib_Poro_three_field}

\end{document}